\documentclass[12pt]{article}
\usepackage[T2A]{fontenc}
\usepackage[utf8]{inputenc}
\usepackage{tikz,bussproofs,amssymb,amsmath,amsthm,MnSymbol}

\EnableBpAbbreviations

\newcommand{\tikzmark}[1]{\tikz[overlay,remember picture] \node (#1) {};}

\theoremstyle{plain}
\newtheorem{theorem}{Theorem}
\newtheorem{lemma}{Lemma}
\newtheorem{propos}{Proposition}
\newtheorem{remark}{Remark}

\author{Daniyar Shamkanov}
\title{Cyclic proofs in the equational version of Primitive recursive arithmetic}
\date{}
\begin{document}
\maketitle
\begin{abstract}
In this brief note, we present a cyclic proof system developed specifically for the equational version of Primitive recursive arithmetic and establish the equivalence of the two systems. A distinctive feature of our approach is that it does not rely on automata-theoretic methods but is implemented primarily using the tools of structural proof theory.
\end{abstract}
\subsection*{Introduction}


Modern proof theory has increasingly turned to cyclic structures as a powerful alternative to traditional finitary proofs in systems that capture various aspects of inductive reasoning (as well as in provability logics). Cyclic proof systems replace explicit induction rules with global soundness conditions, allowing proofs to be represented as finite directed graphs containing cycles. While most such systems for various versions of arithmetic are quite complex, we propose a simple "toy" system, which we hope will pave the way for a better understanding of cyclic proofs in arithmetic.
In this brief note, we present a cyclic proof system developed specifically for the equational version of Primitive recursive arithmetic and establish the equivalence of the two systems. A distinctive feature of our approach is that it does not rely on automata-theoretic methods but is implemented primarily using the tools of structural proof theory.
\subsection*{Primitive recursive arithmetic} Skolem's primitive recursive arithmetic, originated in \cite{Skolem23}, has several formulations in slightly different languages. A variant of this system in a language without logical connectives and quantifiers was given by Curry in \cite{Curry41}, and its refined version was presented by Goodstein in \cite{Goodstein54} (see also \cite{Goodstein57}). In this section, our definition of primitive recursive arithmetic basically follows Goodstein.

In order to define the language of $\mathsf{PRA}$, we recall the definition of primitive recursive terms. The sets $\mathsf{PR}^n$ of $n$-ary primitive recursive function symbols are inductively defined as
\begin{itemize}
    \item $0\in \mathsf{PR}^0$, $o\in \mathsf{PR}^1$, $s \in \mathsf{PR}^1$ and $I^n_k \in \mathsf{PR}^n$ ($1\leq k \leq n$);
    \item if $g \in \mathsf{PR}^m$ for $m>1$ and $f_1,\dotsc, f_m \in \mathsf{PR}^n$, then $C(g, f_1, \dotsc, f_m)\in \mathsf{PR}^n$;
    \item if $g \in \mathsf{PR}^n$ and $h \in \mathsf{PR}^{n+2}$, then $R(g,h)\in \mathsf{PR}^{n+1}$.
\end{itemize}
Primitive recursive terms, or terms of the system $\mathsf{PRA}$, are build from a countable set of variables $\{x_0, x_1, x_2 \dotsc \}$ by means of the function symbols from $\mathsf{PR}:=\bigcup_{n\in \mathbb{N}} \mathsf{PR}^n$. Formulas of $\mathsf{PRA}$ are defined as equations between these terms.

The system $\mathsf{PRA}$ is given by the initial equations $A=A$ and the following inference rules for primitive recursive function symbols: 
\begin{gather*}
    \AXC{$F(0)=A$}
    \LeftLabel{}
    \RightLabel{ ,}
    \UIC{$F(o(B))=A$ }
    \DisplayProof\qquad
    \AXC{$A=F(0)$}
    \LeftLabel{}
    \RightLabel{ ,}
    \UIC{$A=F(o(B))$ }
    \DisplayProof\\\\
    \AXC{$F(B_k)=A$}
    \LeftLabel{}
    \RightLabel{ ,}
    \UIC{$F(I^n_k(B_1,\dotsc, B_n))=A$ }
    \DisplayProof\qquad
    \AXC{$A=F(B_k)$}
    \LeftLabel{}
    \RightLabel{ ,}
    \UIC{$A=F(I^n_k(B_1,\dotsc, B_n))$ }
    \DisplayProof\\\\
    \AXC{$F(g(f_1 (\vec{B}),\dotsc, f_m(\vec{B})))=A$}
    \LeftLabel{}
    \RightLabel{ ,}
    \UIC{$F(C(g,f_1,\dotsc, f_m)(\vec{B}))=A$ }
    \DisplayProof\qquad
    \AXC{$A=F(g(f_1 (\vec{B}),\dotsc, f_m(\vec{B})))$}
    \LeftLabel{}
    \RightLabel{ ,}
    \UIC{$A=F(C(g,f_1,\dotsc, f_m)(\vec{B}))$ }
    \DisplayProof \\\\
    \AXC{$F(g(\vec{B}))=A$}
    \LeftLabel{}
    \RightLabel{ ,}
    \UIC{$F(R(g,h)(\vec{B}, 0))=A$ }
    \DisplayProof\qquad
    \AXC{$A=F(g(\vec{B}))$}
    \LeftLabel{}
    \RightLabel{ ,}
    \UIC{$A=F(R(g,h)(\vec{B}, 0))$}
    \DisplayProof
\end{gather*}
\begin{gather*}
    \AXC{$F(h(\vec{B}, E , R(g,h) (\vec{B}, E)))=A$}
    \LeftLabel{}
    \RightLabel{ ,}
    \UIC{$F(R(g,h)(\vec{B}, s(E)))=A$}
    \DisplayProof\qquad
    \AXC{$A=F(h(\vec{B}, E , R(g,h) (\vec{B}, E)))$}
    \LeftLabel{}
    \RightLabel{ .}
    \UIC{$A=F(R(g,h)(\vec{B}, s(E)))$ }
    \DisplayProof\\
\end{gather*}
In these rules, we assume that $F$ contains a single occurrence of a fresh variable that is replaced by various terms. 

The system $\mathsf{PRA}$ also contains the rules:
\begin{gather*}
    \AXC{$G(x)=H(x)$}
    \LeftLabel{$\mathsf{sub}$}
    \RightLabel{ ,}
    \UIC{$G(A)=H(A)$ }
    \DisplayProof\qquad
    \AXC{$A=B$}
    \LeftLabel{$\mathsf{rep}$}
    \RightLabel{ ,}
    \UIC{$G(A)=G(B)$ }
    \DisplayProof\qquad
    \AXC{$A=B$}
    \AXC{$B=C$}
    \LeftLabel{$\mathsf{tran}$}
    \RightLabel{ ,}
    \BIC{$A=C$ }
    \DisplayProof\\\\
    \AXC{$G(0)=H(0)$}
    \AXC{$G(s(x))=J(x,G(x))$}
    \AXC{$J(x,H(x))=H(s(x))$}
    \LeftLabel{$\mathsf{ind}$}
    \RightLabel{ .}
    \TIC{$G(x)=H(x)$ }
    \DisplayProof
    \end{gather*}

A \emph{proof of an equation $A=B$ in the system $\mathsf{PRA}$} is defined in the usual way as a finite tree of equations constructed according to the rules of $\mathsf{PRA}$ such that any of its leaves is marked by an initial equation and the root is marked by $A=B$. If there is a proof of $A=B$, then the equation $A=B$ is called \emph{provable in $\mathsf{PRA}$}.

\subsection*{Cyclic arithmetical proofs} The system $\mathsf{CPRA}$ is obtained from $\mathsf{PRA}$ by replacing the inference rule ($\mathsf{ind}$) with the rule 
\begin{gather*}
    \AXC{$G(0)=H(0)$}
    \AXC{$G(s(x))=H(s(x))$}
    \LeftLabel{$\mathsf{case}$}
    \RightLabel{ ,}
    \BIC{$G(x)=H(x)$ }
    \DisplayProof
\end{gather*}
where $x$ is called the \emph{active variable} of the given inference. 

A \emph{cyclic proof of an equation $A=B$} is a pair $(\kappa, d)$, where $\kappa $ is a finite tree of equations constructed according to the rules of $\mathsf{CPRA}$ with the root marked by $A=B$ and $d$ is a function with the following properties: (i) the
function $d$ is defined on the set of all leaves of $\kappa$ that are not marked by the initial equations; (ii) the image $d(a)$ of a leaf $a$ lies on the path from the root of $\kappa$ to the leaf $a$
and is not equal to $a$; (iii) there is an application of the rule ($\mathsf{case}$) on the path from $d(a)$ to $a$, and this path intersects the application on the right premise; (iv) there are no applications of the rule ($\mathsf{sub}$) on the path between $d(a)$ and $a$; (v) $a$ and $d(a)$ are marked by the same equations. We also require that (vi) the path from $d(a)$ to $a$ does not intersect any application of the rule ($\mathsf{case}$) on the left premise and (vii) the node $d(a)$, for any leaf $a$, is the conclusion of an application of the rule ($\mathsf{case}$).
If the function $d$ is defined at a leaf $a$, then we say that the nodes $a$ and $d(a)$ are connected by a back-link.

An \emph{equation $A=B$ is provable in the arithmetic $\mathsf{CPRA}$} if there is a cyclic proof of $A=B$.
\begin{remark}
It is easy to show that conditions (vi) and (vii) from the definition of cylic proofs are redundant. In other words, their adoption does not change the class of provable equations. We add these conditions to simplify further study of the system $\mathsf{CPRA}$.
\end{remark}
\begin{propos}
If an equation $A=B$ is provable in $\mathsf{PRA}$, then it is provable in $\mathsf{CPRA}$.
\end{propos}
\begin{proof}
Assume we have a proof $\pi$ of $A=B$ in $\mathsf{PRA}$. We replace every application of ($\mathsf{ind}$) in the proof $\pi$ 
\[
    \AXC{$G(0)=H(0)$}
    \AXC{$G(s(x))=J(x,G(x))$}
    \AXC{$J(x,H(x))=H(s(x))$}
    \LeftLabel{$\mathsf{ind}$}
    \RightLabel{ }
    \TIC{$G(x)=H(x)$ }
    \DisplayProof\]
with the following cyclic derivation
\[\scriptsize
    \AXC{$G(0)=H(0)$}
    \AXC{$G(s(x))=J(x,G(x))$}
    \AXC{$G(x)=H(x)$ \tikzmark{A}}
    \LeftLabel{$\mathsf{rep}$}
    \UIC{$J(x,G(x))=J(x,H(x))$}
    \LeftLabel{$\mathsf{tran}$}
    \BIC{$G(s(x))=J(x,H(x))$}
    \AXC{$J(x,H(x))=H(s(x))$}
    \LeftLabel{$\mathsf{tran}$}
    \BIC{$G(s(x))=H(s(x))$}
    \LeftLabel{$\mathsf{case}$}
    \RightLabel{ }
    \BIC{$G(x)=H(x)$ \tikzmark{B}}
    \DisplayProof
    \begin{tikzpicture}[overlay,remember picture,>=latex,distance=8.0cm] \draw[->, thick](A.north east) to [out=5,in=-10](B.east);
\end{tikzpicture}
\] 
and obtain the required cyclic proof of $A=B$ in $\mathsf{CPRA}$.
\end{proof}

\subsection*{From cyclic proofs to ordinary ones in an extended language} In this section, we consider a formulation of primitive recursive arithmetic in a first-order language with only bounded quantification. Formulas of $\mathsf{PRA}^\prime$ are built from
equations between primitive recursive terms by means of Boolean connectives and bounded quantifiers: if $P$ is a formula of $\mathsf{PRA}^\prime$, $t$ is a primitive recursive term and $x$ is a variable such that $x$ does not occur in $t$, then $\forall x\leqslant t \; P$ is a formula of $\mathsf{PRA}^\prime$.

\begin{theorem}
If an equation $A=B$ is provable in $\mathsf{CPRA}$, then it is provable in $\mathsf{PRA}^\prime$.
\end{theorem} 
\begin{proof}
Assume we have a cyclic proof $\pi=(\kappa, d)$ of $A=B$ in $\mathsf{CPRA}$. We prove that $\mathsf{PRA}^\prime \vdash A=B$ by induction on the height of $\kappa$.

For any node $w$ of $\kappa$, we denote the equation of the node $w$ by $A_w=B_w$ and the subtree of $\kappa$ with the root $w$ by $\kappa_w$. We define $\mathit{rk} (w)$ as the height of the tree obtained from $\kappa_w$ by cutting every branch at the first from the root premise of the rule ($\mathsf{sub}$) and the first from the root premise of the rule ($\mathsf{case}$). In other words, $\mathit{rk} (w)$ is the length of the longest path in the tree $\kappa_w$ that is directed away from the root and does not intersect applications of rules ($\mathsf{sub}$) and ($\mathsf{case}$). For example, if $\kappa_w$ consists only of one node, then $\mathit{rk} (w)=0$.

We define the \emph{main fragment of $\pi=(\kappa, d)$} as a tree obtained from $\pi$ by cutting every branch of $\kappa$ at the first from the root premise of the rule ($\mathsf{sub}$) and the first from the root left premise of the rule ($\mathsf{case}$). 
We denote the set of nodes of the main fragment of $\pi$ by $W$ and the set of conclusions of applications of the rule ($\mathsf{case}$) in the main fragment by $V$. 
 
We also put 
\[Q:=\bigwedge_{v\in V} A_v=B_v.\]

Now we claim that, for any $w\in W$,
\begin{gather}\label{claim1}
\mathsf{PRA}^\prime \vdash Q\to A_w=B_w. 
\end{gather}
We prove the claim applying the induction hypothesis for cyclic proofs $\pi^\prime =(\kappa^\prime, d^\prime)$ with the height of $\kappa^\prime$ being less than the height of $\kappa$. In addition, we argue by subinduction on $\mathit{rk}(w)$. 

Case 1: the tree $\kappa_w$ consists only of an initial equation. In this case, the equation $ A_w=B_w$ has the form $C=C$. Trivially, we have $\mathsf{PRA}^\prime \vdash Q\to A_w=B_w$.

Case 2:  the tree $\kappa_w$ consists only of one leaf, and this leaf is not marked by an initial equation. In this case, the node $w$ is a leaf of $\kappa$ connected with another node $d(w)$ by a back-link. Since $w\in W$, we have $d(w)\in V$. Therefore,  
\[\mathsf{PRA}^\prime \vdash\bigwedge_{v\in V} A_v=B_v\to A_{d(w)}=B_{d(w)}.\]
Since the equation $A_w=B_w$ coincides with $A_{d(w)}=B_{d(w)}$, we immediately obtain $\mathsf{PRA}^\prime \vdash Q\to A_w=B_w$.

Case 3: the tree $\kappa_w$ has the form
\[
    \AXC{$\kappa^\prime$}
    \noLine
    \UIC{$\vdots$}
    \noLine
    \UIC{$G(x)=H(x)$}
    \LeftLabel{$\mathsf{sub}$}
    \RightLabel{ ,}
    \UIC{$G(C)=H(C)$}
    \DisplayProof
\]
where $G(C)=H(C)$ coincides with $A_w=B_w$. 

Since there are no applications of the rule ($\mathsf{sub}$) in between two nodes connected by a back-link, any leaf of $\kappa^\prime$ from the domain of $d$ is connected by a back-link with a node from $\kappa^\prime$. Hence, we have a cyclic proof $\pi^\prime =(\kappa^\prime, d^\prime)$ of $G(x)=H(x)$ in $\mathsf{CPRA}$. Note that the height of $\kappa^\prime$ is less than the height of $\kappa$. Thus, applying the induction hypothesis for $\pi^\prime$, we obtain $\mathsf{PRA}^\prime \vdash G(x)=H(x)$. Consequently, $\mathsf{PRA}^\prime \vdash G(C)=H(C)$ and $\mathsf{PRA}^\prime \vdash Q \rightarrow A_w=B_w$.

Case 4: the tree $\kappa_w$ has the form
\[
    \AXC{$\kappa^\prime$}
    \noLine
    \UIC{$\vdots$}
    \noLine
    \UIC{$G(0)=H(0)$}
    \AXC{$\kappa^{\prime\prime}$}
    \noLine
    \UIC{$\vdots$}
    \noLine
    \UIC{$G(s(x))=H(s(x))$}
    \LeftLabel{$\mathsf{case}$}
    \RightLabel{ ,}
    \BIC{$G(x)=H(x)$}
    \DisplayProof
\]
where $G(x)=H(x)$ coincides with $A_w=B_w$. In this case, $w\in V$. Trivially, we have 
\[\mathsf{PRA}^\prime \vdash\bigwedge_{v\in V} A_v=B_v\to A_{w}=B_{w},\]
i.e. $\mathsf{PRA}^\prime \vdash Q \rightarrow A_w=B_w$.

Case 5. The tree $\kappa_w$ has one of the following forms:
\[
        \AXC{$\kappa^\prime$}
    \noLine
    \UIC{$\vdots$}
    \noLine
    \UIC{$C=D$}
    \LeftLabel{$\mathsf{rep}$}
    \RightLabel{ ,}
    \UIC{$G(C)=G(D)$ }
    \DisplayProof\quad \quad
    \AXC{$\kappa^\prime$}
    \noLine
    \UIC{$\vdots$}
    \noLine
    \UIC{$A_w=D$}
    \AXC{$\kappa^{\prime\prime}$}
    \noLine
    \UIC{$\vdots$}
    \noLine
    \UIC{$D=B_w$}
    \LeftLabel{$\mathsf{tran}$}
    \RightLabel{ ,}
    \BIC{$A_w=B_w$ }
    \DisplayProof
\]
where $G(C)=G(D)$ coincides with $A_w=B_w$. From the subinduction hypotheses for children of $w$ in $\kappa_w$, we see $\mathsf{PRA}^\prime \vdash Q \rightarrow C=D$ ($\mathsf{PRA}^\prime \vdash Q \rightarrow A_w=D$ and $\mathsf{PRA}^\prime \vdash Q \rightarrow D=B_w$). Since $\mathsf{PRA}^\prime \vdash C=D \rightarrow G(C)=G(D)$ and $\mathsf{PRA}^\prime \vdash (A_w=D \wedge D=B_w) \rightarrow A_w=B_w$, we obtain $\mathsf{PRA}^\prime \vdash Q \rightarrow A_w=B_w$ in both cases.

The remaining case, when the equation $A_w=B_w$ is obtained in $\kappa_w$ by one of the inference rules for primitive recursive function symbols, can be easely checked in the same way as case 5, so we omit further details. The claim is proved.

Now recall that the root of the main fragment of $\pi$ is marked by $A=B$. From (\ref{claim1}), we immediately obtain
\begin{gather}\label{assertion1}
\mathsf{PRA}^\prime \vdash Q\to A=B.
\end{gather}

Let $y_1, \dotsc, y_n$ be the list of all active variables of applications of the rule ($\mathsf{case}$) in the main fragment of $\pi$. We set 
\[P(z):= \forall y_1,\dotsc , y_n\leqslant z\; (y_1+\dotsb + y_n=z \to Q),\]
where $\forall y_1,\dotsc , y_n\leqslant z$ is abbreviation for $\forall y_1\leqslant z\; \forall y_2\leqslant z \dotso \forall y_n \leqslant z$.

Now we claim  
\begin{gather}\label{claim2}
\mathsf{PRA}^\prime \vdash P(0), \qquad \mathsf{PRA}^\prime \vdash P(z)\to P(s(z)).
\end{gather}

Notice that, for each $v$ from $V$, the tree $\kappa_v$ has the form
\[
    \AXC{$\kappa^\prime$}
    \noLine
    \UIC{$\vdots$}
    \noLine
    \UIC{$G_v(0)=H_v(0)$}
    \AXC{$\kappa^{\prime\prime}$}
    \noLine
    \UIC{$\vdots$}
    \noLine
    \UIC{$G_v(s(y_j))=H_v(s(y_j))$}
    \LeftLabel{$\mathsf{case}$}
    \RightLabel{ ,}
    \BIC{$G_v(y_j)=H_v(y_j)$}
    \DisplayProof
\]
where $y_j$ is the active variable of the inference, and $G_v(y_j)=H_v(y_j)$ coincides with $A_v=B_v$. Since there are no left premises of the rule ($\mathsf{case}$) in between two nodes connected by a back-link, we have a cyclic proof $\pi^\prime =(\kappa^\prime, d^\prime)$ of $G_v(0)=H_v(0)$ in $\mathsf{CPRA}$. From the induction hypothesis for $\pi^\prime$, we obtain $\mathsf{PRA}^\prime \vdash G_v(0)=H_v(0)$. It follows that $\mathsf{PRA}^\prime \vdash A_v(0,\dotsc, 0)=B_v(0,\dotsc, 0)$, where $A_v =A_v(y_1,  \dotsc, y_n)$ and $B_v =B_v(y_1,  \dotsc, y_n)$.

Consequently, 
\[\mathsf{PRA}^\prime \vdash \bigwedge_{v\in V} A_v(0,\dotsc, 0)=B_v(0,\dotsc, 0) \qquad\text{and}\qquad \mathsf{PRA}^\prime \vdash Q(0,\dotsc, 0),\]
where $Q=Q(y_1,\dotsc, y_n)$. Hence, $\mathsf{PRA}^\prime \vdash P(0)$.

In order to prove that $\mathsf{PRA}^\prime \vdash P(z)\to P(s(z))$, it is sufficient to show 
\[\mathsf{PRA}^\prime \vdash P(z)\to (y_1+\dotsb + y_n=s(z) \to A_v=B_v)\]
for each $v$ form $V$. Let $y_j$ be the active variable corresponding to the node $v$. Arguing in $\mathsf{PRA}^\prime$, we consider two cases: $y_j=0$ or $y_j=s(y^\prime_j)$. If $y_j=0$, then $A_v=B_v$ is equivalent to $G_v(0)=H_v(0)$, which is already provable in $\mathsf{PRA}^\prime$.

Suppose that $y_j=s(y^\prime_j)$, $P(z)$ and 
\begin{gather*}
y_1+\dotsb +y_{j-1}+ s(y^\prime_j)+ y_{j+1}+ \dotsb + y_n= s(z).
\end{gather*}
Then
\begin{gather}\label{assertion2}
y_1+\dotsb +y_{j-1}+ y^\prime_j+ y_{j+1}+ \dotsb + y_n= z.
\end{gather}
and
\begin{gather}\label{assertion3}
y_1,\dotsc, y_{j-1}, y^\prime_j, y_{j+1}, \dotsc , y_n \leqslant z.
\end{gather}
From (\ref{assertion3}), (\ref{assertion2}) and $P(z)$, we obtain $Q(y_1,\dotsc, y_{j-1}, y^\prime_j, y_{j+1}, \dotsc , y_n)$. 

Recall that the node $v$ is the conclusion of an application the rule ($\mathsf{case}$) in the main fragment of $\pi$. Let $w$ be the node corresponding to the right premise of this application. From (\ref{claim1}) and $Q(y_1,\dotsc, y_{j-1}, y^\prime_j, y_{j+1}, \dotsc , y_n)$, we have 
\[A_w(y_1,\dotsc, y_{j-1}, y^\prime_j, y_{j+1}, \dotsc , y_n)=B_w(y_1,\dotsc, y_{j-1}, y^\prime_j, y_{j+1}, \dotsc , y_n).\]
From the definition of the rule ($\mathsf{case}$), this equation coincides with
\[A_v(y_1,\dotsc, y_{j-1}, s(y^\prime_j), y_{j+1}, \dotsc , y_n)=B_v(y_1,\dotsc, y_{j-1}, s(y^\prime_j), y_{j+1}, \dotsc , y_n).\]
We recall that $y_j=s(y^\prime_j)$ and obtain the required equation 
\[A_v(y_1,\dotsc, y_n)=B_v(y_1,\dotsc,  y_n).\]
The second case is checked, and assertion (\ref{claim2}) is established.

Applying the induction rule in $\mathsf{PRA}^\prime$ for (\ref{claim2}), we obtain $\mathsf{PRA}^\prime \vdash P(z)$. Renaming bound variables in $P(z)$ and substituting $y_1+\dotsb + y_n$ for $z$, we see 
\begin{align*}
\mathsf{PRA}^\prime  &\vdash \forall y^{\prime\prime}_1,\dotsc , y^{\prime\prime}_n\leqslant z\; (y^{\prime\prime}_1+\dotsb + y^{\prime\prime}_n=z \to Q(y^{\prime\prime}_1,\dotsc, y^{\prime\prime}_n))  \\
&\vdash \forall y^{\prime\prime}_1,\dotsc , y^{\prime\prime}_n\leqslant y_1+\dotsb + y_n\; (y^{\prime\prime}_1+\dotsb + y^{\prime\prime}_n=y_1+\dotsb + y_n \to Q(y^{\prime\prime}_1,\dotsc, y^{\prime\prime}_n))\\
&\vdash y_1+\dotsb + y_n=y_1+\dotsb + y_n \to Q \quad\ \text{(since $y_i\leqslant y_1+\dotsb + y_n$ for $i\in \{1, \dotsc, n\}$)}\\
&\vdash Q.
\end{align*} 
From (\ref{assertion1}), it follows that $\mathsf{PRA}^\prime \vdash A=B$, which concludes the proof.
\end{proof}

\subsection*{Back to the equational language} Although the system $\mathsf{PRA}^\prime$ is given in the firs-order language with bounded quantification, it proves precisely the same equations as the original system $\mathsf{PRA}$. 
\begin{propos}\label{proposition2}
If an equation $A=B$ is provable in $\mathsf{PRA}^\prime$, then it is provable in $\mathsf{PRA}$.
\end{propos}
This result is obtained by means of the follwing translation. For any formula $P$ of the language of $\mathsf{PRA}^\prime$, the primitive recursive term $T_P$ is inductively defined as: $T_\bot := s(0)$, $T_{A=B}= (A\dotminus B)+ (B \dotminus A)$, $T_{Q_0 \to Q_1}:=(1 \dotminus T_{Q_0}) \cdot T_{Q_1}$ and
\[
 T_{\forall x\leqslant t\; Q}:= \sum\limits_{0\leqslant x \leqslant t} T_{Q}(x).
\]
The following two lemmata are established along the lines of \cite{Goodstein57} and \cite{Schwartz87}, so we omit the proofs.
\begin{lemma}
If a formula $P$ is provable in $\mathsf{PRA}^\prime$, then the equation $T_P=0$ is provable in $\mathsf{PRA}$.
\end{lemma}
\begin{lemma}
If an equation $T_{A=B}=0$ is provable in $\mathsf{PRA}$, then $A=B$ is provable in $\mathsf{PRA}$.
\end{lemma}
Now Proposition \ref{proposition2} is established. Moreover, we see that the systems $\mathsf{PRA}$ and $\mathsf{CPRA}$ are equivalent.
\begin{theorem}
For any equation $A=B$, we have
\[\mathsf{PRA} \vdash A=B \Longleftrightarrow \mathsf{CPRA}\vdash A=B.\]
\end{theorem}

\end{document}